\documentclass[12pt]{amsart}
\usepackage{amsmath,latexsym,amssymb}
\usepackage{enumerate}
\usepackage{amsthm}
\usepackage{epsfig,graphicx}
\usepackage{tikz}
\usepackage{mathrsfs}

\newtheorem{theorem}{Theorem}[section]
\newtheorem{definition}[theorem]{Definition}
\newtheorem{corollary}[theorem]{Corollary}

\newtheorem{lemma}[theorem]{Lemma}
\newtheorem{remark}[theorem]{Remark}
\newtheorem*{remarks*}{Remarks}
\newtheorem{proposition}[theorem]{Proposition}

\newcommand\NN{\mathbb{N}}
\newcommand\QQ{\mathbb{Q}}

\newcommand\ZZ{\mathbb{Z}}
\newcommand\eps{\varepsilon}
\newcommand\diam{\mathrm{diam}}

\numberwithin{equation}{section}

\title[Numbers with matching continued fraction and decimal expansions]{On the existence of numbers with matching continued fraction and decimal expansions}

\author{Pieter Allaart}
\thanks{The first author is partially supported by Simons Foundation grant \# 709869. 
The second author is partially supported by NSF grant DMS-1800323. }
\address[P. Allaart]{Mathematics Department, University of North Texas, 1155 Union Cir \#311430, Denton, TX 76203-5017, U.S.A.}
\email{allaart@unt.edu}

\author{Stephen Jackson}
\address[S. Jackson]{Mathematics Department, University of North Texas, 1155 Union Cir \#311430, Denton, TX 76203-5017, U.S.A.}
\email{Stephen.Jackson@unt.edu}

\author{Taylor Jones}
\address[T. Jones]{Mathematics Department, University of North Texas, 1155 Union Cir \#311430, Denton, TX 76203-5017, U.S.A.}
\email{RandallJones2@my.unt.edu}

\author{David Lambert}
\address[D. Lambert]{Mathematics Department, University of North Texas, 1155 Union Cir \#311430, Denton, TX 76203-5017, U.S.A.}
\email{DavidLambert2@my.unt.edu}

\begin{document}

\begin{abstract}
A Trott number is a number $x\in(0,1)$ whose continued fraction expansion is equal to its base $b$ expansion for a given base $b$, in the following sense: If $x=[0;a_1,a_2,\dots]$, then $x=(0.\hat{a}_1\hat{a}_2\dots)_b$, where $\hat{a}_i$ is the string of digits resulting from writing $a_i$ in base $b$. In this paper we characterize the set of bases for which Trott numbers exist, and show that for these bases, the set $T_b$ of Trott numbers is a complete $G_\delta$ set. We prove moreover that the union $T:=\bigcup_{b\geq 2} T_b$ is nowhere dense and has Hausdorff dimension less than one. Finally, we give several sufficient conditions on bases $b$ and $b'$ such that $T_b\cap T_{b'}=\emptyset$, and conjecture that this is the case for all $b\neq b'$. This question has connections with some deep theorems in Diophantine approximation.
\end{abstract}


\maketitle

\section{Introduction}

An area of general interest in number theory, dynamics, and numeration systems concerns the 
study of the itinerary of a point in a dynamical system. Given a Polish space 
$X$ (a complete, separable, metric space), a continuous map $T\colon X\to X$, and a partition 
$\{ X_i\}_{i \in \mathcal{D}}$ of $X$ into a finite or countably infinite number of pieces, 
each $x \in X$ produces its forward orbit $(T^0(x),T^1(x),T^2(x),\dots)$, where $T^0(x)=x$. 
The itinerary of $x$ is the sequence $(i_0,i_1,i_2,\dots)$ where $T^k(x)\in X_{i_k}$. 
The index set $\mathcal{D}$ of the partition is the digit set for the itinerary. 
Most numeration systems for real numbers arise in this manner, such as base-$b$ expansions,
continued fractions, $\beta$-expansions, etc.\ One general question which arises is to what extent
can the itineraries of a point under two different dynamical systems on $X$ be similar? 
For this paper, the two numeration systems of interest are continued fractions and base-$b$ 
expansions. These systems, however,  have different digit sets, with continued fractions using 
$\mathcal{D}=\mathbb{N}:=\{ 1,2,3,\dots\}$, and base-$b$ expansions using $\mathcal{D}=\{ 0,1,\dots, b-1\}$. 
So, we must adopt a reasonable convention as to how we compare two such expansions. This will be done 
in a natural manner, by writing each digit $a_i$ in the continued fraction expansion in base-$b$
and then concatenating these strings (we give the precise definition below). 
The question of particular interest for us is whether a real number $x \in (0,1)$ can have 
a continued fraction expansion and a base-$b$ expansion which match in this manner. 

This question, or actually a slight variation of it, was first asked by M.\ Trott in 1999. 
Trott asked whether there exist numbers whose continued fraction expansion equals their decimal expansion. In other words, is there a sequence $a_1,a_2,\dots$ of single-digit numbers such that
\[
[0;a_1,a_2,a_3,\dots]:=\cfrac{1}{a_1+\cfrac{1}{a_2+\cfrac{1}{a_3+\dots}}}=0.a_1 a_2 a_3\dots?
\]
In 2006 \cite{Trott} he published his discovery,
\[
x=[0;1,0,8,4,1,0,1,5,1,2,2\dots]=0.10841015122\dots,
\]
along with a computer algorithm for finding such numbers.
Of course, use of the digit zero is normally not allowed in continued fraction expansions, and is rather unnatural from a dynamical point of view. Nonetheless, the number $x$ above is now known as {\em Trott's constant} \cite{OEIS}. 

In this paper, we tweak Trott's original question and ask for solutions without the digit zero, but instead allowing multi-digit partial quotients. (It is not hard to see that no such example exists if one requires that $a_i\in\{1,2,\dots,9\}$ for every $i$.) For instance, in base $10$ it appears that one could start with
$$x=[0;3,29,54,78,\dots]=0.3295478\dots.$$
This problem has led to some surprisingly rich and intricate mathematics. While we were able to construct some examples of such numbers (which, in honor of M. Trott, we call {\em Trott numbers} here), we did so only after an initial struggle, due mainly to the incompatibility of the dynamical systems that generate the continued fraction expansion and decimal expansion, respectively. We found that Trott numbers exist in some bases, but not in others. Many interesting questions remain unanswered.


We begin with a formal definition. For definiteness, $\NN=\{1,2,3,\dots\}$ will denote the set of all positive integers in this article.

\begin{definition}
Let $b\in\NN$ with $b\geq 2$. We call a number $x$ a {\em Trott number} in base $b$ if $x$ has an infinite continued fraction expansion $x=[0;a_1,a_2,a_3,\dots]=(0.\hat{a}_1\hat{a}_2\hat{a}_3\dots)_b$, where $\hat{a}_i$ is the string of digits resulting from writing $a_i$ in base $b$.
\end{definition}

It follows immediately that Trott numbers, if they exist, cannot be rational or quadratic irrational. At first glance, it appears easy to construct Trott numbers.
Take $b=10$ and start with $[0;3,29]=0.329\overline{54}$. Then we simply take as many subsequent digits as desired to be the next term in the continued fraction.
For instance, $[0;3,29]$ could be extended to $[0;3,29,545]=0.3295456\dots$, and so on. 
The problem with this simple approach is that some of these  numbers fail to admit a continuation. For example, the number $[0;3,29,5,7]=0.3295703\dots$ can not be extended, as $0.3295703<[0;3,29,5,7,a]<0.3295705$ for every $a\in\NN$.
Zeros appearing in the decimal expansion are dealt with by being absorbed as a non-leading digit in the continued fraction expansion.
So in our example, even though $[0;3,29,5,7]$ can not be extended, $[0;3,29,5]$ can, since $[0;3,29,5,710]=0.32957109\dots$. However, to guarantee that such a construction can be continued indefinitely requires some ingenuity, because of the possibility of very long strings of zeros eventually appearing in the expansion of $[0;a_1,\dots,a_n]$.

Let $T_b$ denote the set of all Trott numbers in base $b$. Our first main result is:

\begin{theorem} \label{thm:main}
There exists a Trott number in base $b$ if and only if
$$b\in\Gamma:=\{3\}\cup\bigcup_{k=1}^\infty \{k^2+1,k^2+2,\dots,k^2+k\}.$$
Furthermore, if $T_b$ is nonempty, then it is uncountable.
\end{theorem}

Thus, the first several bases that admit Trott numbers are $2$, $3$, $5$, $6$, $9$, $10$, $11$, $12$, $17$, $18$, $19$, $20$, $26$, $27,\dots$.
We suspect that $T_b$ actually has positive Hausdorff dimension for $b\in\Gamma$. Unfortunately, our construction falls well short of allowing us to draw this stronger conclusion. We are, however, able to classify the set $T_b$ with respect to the Borel hierarchy:

\begin{theorem} \label{thm:G-delta}
For each $b\in\Gamma$, $T_b$ is a complete $G_\delta$ set. (That is, $T_b$ is $G_\delta$ but not $F_\sigma$.) 
\end{theorem}

Our next result shows, in two different ways, that the set of numbers which are Trott in {\em any} base is quite small.

\begin{theorem} \label{thm:upper-bound}
Let $T:=\bigcup_{b\geq 2} T_b$. Then $T$ is nowhere dense and $\dim_H T<1$, where $\dim_H T$ denotes the Hausdorff dimension of $T$.
\end{theorem}

In fact, our proof yields a simple equation involving the Riemann zeta function which can be solved numerically for any given $b$ to obtain a more precise upper bound for $\dim_H T_b$. In the limit as $b\to\infty$, these bounds decrease to about $0.8643$. We believe that in actual fact, $\dim_H T_b$ must tend to zero as $b\to\infty$, but this appears to be significantly harder to prove.

Finally, we investigate whether a number can ever be Trott in more than one base at once. We believe that the answer is negative, but have only been able to prove several partial results in this direction, which we collect in the following theorem. 

\begin{theorem} \label{thm:multiple-Trott}
Let $b, c\in\Gamma$ with $b<c$, and suppose at least one of the following holds:
\begin{enumerate}[(i)]
\item $\lfloor \sqrt{b} \rfloor \neq \lfloor \sqrt{c} \rfloor$;
\item $\gcd(b,c)>1$;
\item $c=b+1$;
\item $c=b+2$ and there exists $k\geq 3$ such that
\[
k^2<b\leq k^2+\sqrt{\frac{2k^2}{k-2}+1}-1;
\]
\item There exists $k\geq 2$ such that $b=k^2+1$ and $c=k^2+k$;
\item $b>1.185\times 10^{29}$.
\end{enumerate}
Then $T_b\cap T_c=\emptyset$.
\end{theorem}

This paper is organized as follows. Section \ref{sec:non-existence} shows that $T_b=\emptyset$ when $b\not\in\Gamma$. Section 3 completes the proof of Theorem \ref{thm:main} by showing that $T_b\neq\emptyset$ for $b\in\Gamma$. Theorem \ref{thm:G-delta} is proved in Section \ref{sec:Borel}, and Theorem \ref{thm:upper-bound} in Section \ref{sec:upper-bound}. Finally, Theorem \ref{thm:multiple-Trott} is proved in Section \ref{sec:multiple-Trott}.

\section{Proof of non-existence for $b\not\in\Gamma$} \label{sec:non-existence}

A brief moment of reflection shows that the expansion of any Trott number must begin with $a_1=\lfloor\sqrt{b}\rfloor$. When $b=k^2$ is a perfect square, this implies that no Trott numbers exist because numbers whose continued fraction expansion begins with $k$ lie in $[1/(k+1),1/k]$, while numbers whose base $b$ expansion begins with $k$ lie in $[\frac{k}{b},\frac{k+1}{b}]=[\frac{1}{k},\frac{k+1}{k^2}]$, leaving only the {\em finite} expansion $[0;k]=0.k$.

Now fix $k\geq 2$ and suppose $b\in\{k^2+k+1,\dots,(k+1)^2\}$. The case $b=(k+1)^2$ was addressed above, so let $b<(k+1)^2$. We will show that there is no suitable choice for $a_2$. Note that for $j\in\NN$, the interval of numbers whose continued fraction expansion begins with $[0;k,j,\dots]$ is
$$I_j:=\left[\frac{1}{k+\frac{1}{j}},\frac{1}{k+\frac{1}{j+1}}\right].$$
If $j>\frac{b-1}{k}$, then 
$$\frac{1}{k+\frac{1}{j}}>\frac{b-1}{bk}\geq \frac{k+1}{b}$$
since $b\geq k(k+1)+1$, and so the base $b$ expansion of any number in $I_j$ begins with $0.(k+1)$. Therefore we must have $j\leq \frac{b-1}{k}$. We claim that this does not work either, since
\begin{equation} \label{eq:wrong-endpoints}
\frac{1}{k+\frac{1}{j}}>\frac{k}{b}+\frac{j+1}{b^2},
\end{equation}
so the second base $b$ digit of any number in $I_j$ is at least $j+1$.
For $j=1$, \eqref{eq:wrong-endpoints} is equivalent to $b^2>(k+1)(bk+2)$, or equivalently, $b\{b-k(k+1)\}>2k+2$. Since $b>k(k+1)$, the left hand side is increasing in $b$, and we get
$$b\{b-k(k+1)\}\geq k(k+1)+1=k^2+k+1>2k+2,$$
which holds for all $k\geq 2$.

Next, let 
$$f(x):=\frac{1}{k+\frac{1}{x}}-\frac{k}{b}-\frac{x+1}{b^2}=\frac{x}{kx+1}-\frac{k}{b}-\frac{x+1}{b^2},$$
and observe that
$$f'(x)=\frac{1}{(kx+1)^2}-\frac{1}{b^2}\geq 0 \qquad\mbox{for}\ x\leq\frac{b-1}{k}.$$
Hence $f$ is increasing on $[1,\frac{b-1}{k}]$, so for all $j$ in this interval, $f(j)\geq f(1)>0$, proving \eqref{eq:wrong-endpoints}.

\section{Proof of existence for $b\in\Gamma$} \label{sec:existence}

\subsection{The cases $b=2$ and $b=3$}

The case $b=3$ is special, since it does not fall in the range $\{k^2+1,\dots,k^2+k\}$ for any $k$. Nonetheless, there exist Trott numbers in base 3. (Note that the non-existence argument from the previous section breaks down when $k=1$.)
The case $b=2$ is also special, since our general proof below does not work for $k=1$. We give a separate argument here that addresses these two special bases. In hypothesis (ii) below, $d^l$ denotes the $l$-fold concatenation of the digit $d$ with itself.

\begin{lemma} \label{lem:once-started}
Let $b\geq 2$, and suppose there exist integers $n\geq 2$ and $l\geq 2$ and a finite sequence $(a_1,\dots,a_n)$ of positive integers such that:
\begin{enumerate}[(i)]
\item The expansion of $x_n:=[0;a_1,\dots,a_n]$ in base $b$ begins with $0.\hat{a}_1\dots\hat{a}_n$, and the digit immediately following this string is not zero;
\item The string $\hat{a}_1\dots\hat{a}_n$ contains a word $0w$, where $w$ has length $l$, $w$ does not begin with the digit $0$, $w\neq 1(0^{l-1})$ and $w\neq (b-1)^l$;
\item $q_n^2>b^{s_n+l+1}$, where $x_n=p_n/q_n$ in lowest terms and $s_n$ denotes the length (i.e. number of digits) of the string $\hat{a}_1\dots\hat{a}_n$;
\item $\gcd(q_{n-1},b)=\gcd(q_n,b)=1$.
\end{enumerate}
Then there exist uncountably many Trott numbers in base $b$.
\end{lemma}

\begin{proof}
Fix a value $n_0$ for $n$ and an integer $l$ satisfying (i)-(iv) in the lemma.
Suppose statements (i)-(iv) hold for some $n\geq n_0$. We will now show that $a_{n+1}$ can be constructed so that (i)-(iv) hold also for $n+1$ in place of $n$. The lemma then follows by induction.

By (iv), the expansion of $x_n$ is purely periodic, i.e. of the form $x_n=0.\overline{d_1\dots d_m}$ for some word $d_1\dots d_m$. Hence by (i), it contains the initial string $\hat{a}_1\dots\hat{a}_{n_0}$ infinitely many times. In particular, by (ii), it contains the word $0w$ infinitely many times. In other words, writing $x_n$ as $x_n=0.c_1 c_2 c_3\dots$, there are infinitely many choices of $j\geq 2$ so that 
\begin{equation} \label{eq:word-match}
c_{s_n+j}c_{s_n+j+1}\dots c_{s_n+j+l}=0w.
\end{equation}
For any such $j$, we can take $\hat{a}_{n+1}=c_{s_n+1}\dots c_{s_n+j}$, that is,
$$a_{n+1}=\sum_{i=1}^j c_{s_n+i}b^{j-i}.$$
This is a valid choice of $a_{n+1}$, since $c_{s_n+1}\neq 0$ by (i). Furthermore, since $c_{s_n+j}=0$, $b$ divides $a_{n+1}$. Hence by (iv) and the recursion $q_{n+1}=a_{n+1}q_n+q_{n-1}$, $\gcd(q_{n+1},b)=1$. So (iv) holds for $n+1$ in place of $n$.

Next, we check (iii) for $n+1$ in place of $n$:
$$q_{n+1}^2>a_{n+1}^2 q_n^2\geq b^{2(j-1)}b^{s_n+l+1}=b^{s_n+2j+l-1}=b^{s_{n+1}+j+l-1}\geq b^{s_{n+1}+l+1},$$
where the second inequality follows since $a_{n+1}$ has $j$ digits, and the last inequality uses that $j\geq 2$. 

Finally, we observe that (iii) implies
\begin{align*}
|x_{n+1}-x_n|&=\left|\frac{p_{n+1}}{q_{n+1}}-\frac{p_n}{q_n}\right|=\frac{1}{q_n q_{n+1}}<\frac{1}{a_{n+1}q_n^2}\\
&<b^{-(j-1)}b^{-(s_n+l+1)}=b^{-(s_n+j+l)}=b^{-(s_{n+1}+l)}.
\end{align*}
This, together with \eqref{eq:word-match} and (ii), implies that the expansion of $x_{n+1}$ begins with $0.\hat{a}_1\dots\hat{a}_{n+1}$, and the digit immediately following this string (i.e. the $(s_{n+1}+1)$th) is not zero. Hence, we have (i) for $n+1$ in place of $n$. Of course (ii) holds trivially for $n+1$ in place of $n$ as well.

\bigskip
The induction step shows that the process can be continued indefinitely. Moreover, since at each stage $n$ we have infinitely many choices of $j$ (and hence of $a_{n+1}$), our procedure can generate continuum many Trott numbers.
\end{proof}

For $b=2$, we take $(a_1,a_2,a_3,a_4,a_5)=(1,4,13,36,5)$, or $(1,100,1101,100100,101)$ when written in base 2. Writing $x_n=p_n/q_n=[0;a_1,\dots,a_n]$, a direct computation shows $q_4=2381$ and $q_5=11971$. Furthermore, the binary expansion of $x_n$ begins with $0.1100110110010010111\ldots$. Hence the hypotheses (i), (ii) and (iv) of Lemma \ref{lem:once-started} are satisfied with $n=5$, $l=3$ and $w=110$. 
Since $s_5=17$, condition (iii) is also easily verified.

For $b=3$, we take $(a_1,a_2,a_3,a_4,a_5)=(1,1,44,144,4)$, or $(1,1,1122,12100,11)$ when written in base 3. A direct calculation shows that the ternary expansion of $[0;a_1,\dots,a_5]$ begins with $0.111122121001122\ldots$. Here $q_4=12818\equiv 2\ (\!\!\!\!\mod 3)$ and $q_5=51361\equiv 1\ (\!\!\!\!\mod 3)$. Hence hypotheses (i), (ii) and (iv) of Lemma \ref{lem:once-started} are satisfied with $n=5$, $l=2$ and $w=11$. Since $s_5=13$, condition (iii) is also easily checked.

\begin{remark}
{\rm
The simple method above can be used also for many other bases. However, finding a suitable initial string is a matter of trial and error, and there seems to be no clear general prescription for doing so in an arbitrary base $b$. Therefore, a somewhat more intricate argument is needed for the general case below.
}
\end{remark}

\subsection{The general case}

Fix $k\geq 2$ and let $k^2<b\leq k^2+k$. We have seen that $a_1=k$ is forced. It turns out that the key to proving that an infinite expansion $[0;a_1,a_2,\dots]$ can be constructed, is to choose $a_2$ very carefully. 

First, $1/k$ has an eventually periodic, non-terminating expansion in base $b$, which we write as
\begin{align}
\begin{split}
\frac{1}{k}&=0.c_1\dots c_l \overline{d_1\dots d_m}\\
&=\frac{1}{b^l}\left(\sum_{i=1}^l c_i b^{l-i}+\frac{1}{b^m-1}\sum_{i=1}^m d_i b^{m-i}\right).
\end{split}
\label{eq:base-b-expansion}
\end{align}
For example, if $b=18$ then $k=4$, and $\frac14=0.49=0.48\overline{17}$, where the reader should keep in mind that ``17" is a single digit.
We can rearrange the above as
\begin{equation} \label{eq:clear-fractions}
b^l(b^m-1)=k\left((b^m-1)\sum_{i=1}^l c_i b^{l-i}+\sum_{i=1}^m d_i b^{m-i}\right).
\end{equation}
Note that $m$ need not be the minimal period of $1/k$ in base $b$; we can choose $m$ to be any multiple of the minimal period, hence we can choose $m$ as large as we wish. In the same way, there is no unique choice for the number $l$. We choose $l$ so that $l\geq 1$.

Next, let
\begin{equation} \label{eq:j}
j:=\left\lfloor \sum_{i=1}^l c_i b^{l-i}+\frac{k}{b-k^2}-\frac{b^l-1}{k}+1\right\rfloor.
\end{equation}
This may look complicated, so we first prove a few things about this number $j$.

\begin{lemma} \label{lem:j}
We have $1\leq j\leq k+1$. Moreover, $j=k+1$ if and only if $b=k^2+1$.
\end{lemma}

\begin{proof}
Since \eqref{eq:base-b-expansion} is the non-terminating expansion of $1/k$, we have
\begin{equation} \label{eq:one-over-k-interval}
\sum_{i=1}^l c_i b^{-i}<\frac{1}{k}\leq \sum_{i=1}^l c_i b^{-i}+b^{-l},
\end{equation}
and so
\begin{equation} \label{eq:important-bounds}
1\leq b^l-k\sum_{i=1}^l c_i b^{l-i}\leq k.
\end{equation}
The inequality on the right gives
$$k\sum_{i=1}^l c_i b^{l-i}+\frac{k^2}{b-k^2}-(b^l-1)\geq -k+1+\frac{k^2}{b-k^2}\geq 1,$$
using that $b\leq k^2+k$. Hence \eqref{eq:j} shows $j\geq 1$. On the other hand, the inequality on the left of \eqref{eq:important-bounds} gives
$$k\sum_{i=1}^l c_i b^{l-i}+\frac{k^2}{b-k^2}-(b^l-1)\leq \frac{k^2}{b-k^2}\leq k^2,$$
since $b\geq k^2+1$. Dividing by $k$ and using \eqref{eq:j} it follows that $j\leq k+1$. We see also that equality holds if and only if $b=k^2+1$, because in that case we have $1/k=0.\overline{k}$ so we can take $l=1$, $c_1=k$ and \eqref{eq:j} gives $j=k+1$.
\end{proof}

The next lemma is not strictly needed to prove the main theorem, but it provides useful additional information.

\begin{lemma} \label{lem:no-small-digits}
For $k^2<b<(k+1)^2$, all digits in the non-terminating base $b$ expansion of $1/k$ are at least $k$.
\end{lemma}

\begin{proof}
Write the non-terminating expansion as $1/k=0.c_1 c_2 c_3\dots$. Suppose $c_j<k$ for some $j$. Then
$$\sum_{i=1}^{j-1}\frac{c_i}{b^i}+\frac{c_j}{b^j}<\frac{1}{k}\leq \sum_{i=1}^{j-1}\frac{c_i}{b^i}+\frac{c_j+1}{b^j},$$
and so
$$kc_j<b^j-k\sum_{i=1}^{j-1}c_i b^{j-i}\leq k(c_j+1).$$
But this is impossible, since the middle expression is a multiple of $b$, and $k(c_j+1)\leq k^2<b$.
\end{proof}

We now take
\begin{equation} \label{eq:a2}
a_2=\sum_{i=2}^l c_i b^{m+l-i}+\sum_{i=1}^m d_i b^{m-i}-j,
\end{equation}
with $j$ given by \eqref{eq:j}. Note by Lemma \ref{lem:no-small-digits} that, unless $b=k^2+1$, $j\leq k\leq d_m$ so subtracting $j$ only changes the last digit of $a_2$; in other words, $a_2$ has digits $c_2\dots c_l d_1\dots d_{m-1} (d_m-j)$. When $b=k^2+1$, however, the situation is slightly different. For instance, if $b=10$, then $k=3$ and $1/k=0.\overline{3}$, so we get $\hat{a}_2=33\dots 329$, where the number of 3's depends on the choice of $l$ and $m$.

Let
$$x_2:=\frac{p_2}{q_2}:=[0;a_1,a_2]=[0;k,a_2]=\frac{a_2}{ka_2+1}.$$
We must show that the first $l+m$ digits of $x_2$ are $c_1\dots c_l d_1\dots d_{m-1}(d_m-j)$ (appropriately modified as pointed out above in case $b=k^2+1$). We claim that, more strongly,
\begin{equation} \label{eq:crucial-sandwich}
\frac{k}{b}+\frac{a_2}{b^{l+m}}+\sum_{i=1}^\infty \frac{1}{b^{l+m+i}}<\frac{p_2}{q_2}<\frac{k}{b}+\frac{a_2+1}{b^{l+m}}
\end{equation}
for all sufficiently large $m$.

First, combining \eqref{eq:clear-fractions} and \eqref{eq:a2} and noting that $c_1=k$ gives
\begin{equation} \label{eq:a_2-identity}
b^l(b^m-1)=k\left(a_2+kb^{l+m-1}-\sum_{i=1}^l c_i b^{l-i}+j\right),
\end{equation}
so that
\begin{equation} \label{eq:q2}
q_2=ka_2+1=b^l(b^m-1)-k^2 b^{l+m-1}+k\sum_{i=1}^l c_i b^{l-i}-kj+1.
\end{equation}
It follows that
$$\lim_{m\to\infty}\frac{q_2}{b^{l+m}}=1-\frac{k^2}{b}=\frac{b-k^2}{b}.$$
Hence we obtain
\begin{align*}
b^{l+m}\left(\frac{a_2}{ka_2+1}-\frac{k}{b}-\frac{a_2}{b^{l+m}}\right)
&=\frac{b^{l+m}}{k}\left(1-\frac{1}{ka_2+1}\right)-(a_2+kb^{l+m-1})\\
&=\frac{b^{l+m}}{k}\left(1-\frac{1}{ka_2+1}\right)-\sum_{i=1}^l c_i b^{l-i}+j-\frac{b^l(b^m-1)}{k}\\
&=\frac{b^l}{k}-\frac{b^{l+m}}{kq_2}-\sum_{i=1}^l c_i b^{l-i}+j\\
&\to\frac{b^l}{k}-\frac{b}{k(b-k^2)}-\sum_{i=1}^l c_i b^{l-i}+j
\end{align*}
as $m\to\infty$, where the second equality follows from \eqref{eq:a_2-identity}. By \eqref{eq:j},
$$j>\sum_{i=1}^l c_i b^{l-i}+\frac{k}{b-k^2}-\frac{b^l-1}{k},$$
so that
\begin{align*}
\frac{b^l}{k}-\frac{b}{k(b-k^2)}-\sum_{i=1}^l c_i b^{l-i}+j
&>\frac{1}{k}-\frac{b}{k(b-k^2)}+\frac{k}{b-k^2}=0.
\end{align*}
We claim that the limit is in fact greater than $\sum_{i=1}^\infty b^{-i}=\frac{1}{b-1}$. To see this, suppose by way of contradiction that
$$0<\frac{b^l}{k}-\frac{b}{k(b-k^2)}-\sum_{i=1}^l c_i b^{l-i}+j\leq \frac{1}{b-1}.$$
Writing $A:=j-\sum_{i-1}^l c_i b^{l-i}$ and clearing of fractions, we then obtain
$$0<b^l(b-k^2)-b+k(b-k^2)A\leq\frac{k(b-k^2)}{b-1}.$$
However, the middle expression is an integer, and
$$\frac{k(b-k^2)}{b-1} \leq \frac{k^2}{k^2+k-1}<1,$$
since $k\geq 2$ and $b\leq k^2+k$. This contradiction yields the first inequality in \eqref{eq:crucial-sandwich}, for all sufficiently large $m$.

In a similar way, we can calculate
\begin{align*}
b^{l+m}\left(\frac{k}{b}+\frac{a_2+1}{b^{l+m}}-\frac{a_2}{ka_2+1}\right)
&=\frac{b^{l+m}}{kq_2}+\sum_{i-1}^l c_i b^{l-i}-\frac{b^l}{k}-j+1\\
&\geq \frac{b^{l+m}}{kq_2}-\frac{k}{b-k^2}-\frac{1}{k}
=\frac{b}{k}\left(\frac{b^{l+m-1}}{q_2}-\frac{1}{b-k^2}\right).
\end{align*}
Here taking the limit as $m\to\infty$ does not work since we would get zero. But the last expression is positive for {\em every} $m$, since by \eqref{eq:q2},
\begin{equation} \label{eq:positive-difference}
b^{l+m-1}(b-k^2)-q_2=b^l-k\sum_{i-1}^l c_i b^{l-i}+kj-1\geq kj>0,
\end{equation}
where the first inequality follows from \eqref{eq:important-bounds}, and the last inequality follows since $j\geq 1$ by Lemma \ref{lem:j}. This gives the second inequality in \eqref{eq:crucial-sandwich}.

\bigskip
Having established \eqref{eq:crucial-sandwich}, we next impose a few more technical conditions on $m$. To begin, let
$$b=P_1^{e_1}P_2^{e_2}\cdots P_r^{e_r}$$
be the prime factorization of $b$. We can write \eqref{eq:q2} as $q_2=A+B$, where
$$A=b^{l+m}-k^2 b^{l+m-1}, \qquad B=k\sum_{i=1}^l c_i b^{l-i}-b^l-kj+1.$$
Note that $-B$ is the expression from \eqref{eq:positive-difference}, so $B\neq 0$. Moreover, $b^m|A$ since $l\geq 1$. 
We now choose $m$ large enough so that 
$$\min_{1\leq i\leq r} P_i^{e_i m}>|B|.$$
This ensures that
\begin{equation} \label{eq:no-division}
P_i^{e_i m}\!\nmid q_2, \qquad i=1,\dots,r.
\end{equation}
We now claim that this implies that $x_2=p_2/q_2$ is not a $b$-adic rational, and the pre-periodic part of its expansion has at most $m$ digits. To see this, suppose by way of contradiction that $x_2=z/b^t$ for some $z\in\ZZ$ and $t\in\NN$. By \eqref{eq:crucial-sandwich}, $t>l+m$. Furthermore, $q_2$ can only have prime factors from the list $P_1,\dots,P_r$. By \eqref{eq:no-division}, this implies that $q_2|b^m$ and so $t\leq m$. This contradiction establishes the first part of the claim, which in turn implies that the base $b$ expansion of $x_2$ is eventually periodic. 

Now setting $w:=\lfloor b^m p_2/q_2\rfloor$, we can write
\begin{equation} \label{eq:x2-decomposed}
x_2=\frac{p_2}{q_2}=\frac{w}{b^m}+\frac{p'}{b^m q'},
\end{equation}
for certain integers $p'$ and $q'$ with $\gcd(p',q')=1$. Note that $w$ represents the first $m$ digits of $x_2$, and $p'/q'$ is the rational formed by the remaining digits. We argue that $\gcd(q',b)=1$, and therefore the expansion of $p'/q'$ is purely periodic. First write $q_2=u\tilde{q}_2$ with $\tilde{q}_2$ being the largest divisor of $q_2$ that is relatively prime with $b$. Then $u|b^m$ by \eqref{eq:no-division}, so we can write $b^m=uv$ with $v\in\ZZ$. Clearing fractions in \eqref{eq:x2-decomposed} and dividing by $u$ we obtain $p_2 vq'=\tilde{q}_2(wq'+p')$. It follows that, if $d:=\gcd(q',b)>1$, then $d|wq'+p'$ and hence $d|p'$, contradicting that $\gcd(p',q')=1$. Hence, $d=1$. This proves the second part of the claim.

Therefore, setting $s_2:=l+m$, the expansion of $p_2/q_2$ is periodic at least from the $s_2$-th digit on.

The last precise detail is that we choose $m$ so large that $s_2=l+m\geq 7$. Since
$$q_2>ka_2\geq b^{l+m-2}=b^{s_2-2},$$
we then get
$$q_2^2>b^{2s_2-4}\geq b^{s_2+3}.$$

By choosing $m$ so that all the above conditions are satisfied, we have provided a basis for the induction step, which we describe next.

\bigskip
Let $n\geq 2$, and suppose partial quotients $a_1,a_2,\dots,a_n$ have been constructed so that the following conditions are satisfied:
\bigskip
\begin{enumerate}[(i)]
\item The rational number $x_n:=[0;a_1,a_2,\dots,a_n]$ does not have a terminating base $b$ expansion, and the expansion of $x_n$ begins with the string $0.\hat{a}_1\hat{a}_2,\dots,\hat{a}_n$;
\item With $s_n$ being the length of the string $\hat{a}_1\hat{a}_2,\dots,\hat{a}_n$, the expansion of $x_n$ is periodic from the $s_n$th digit on;
\item With $x_n=p_n/q_n$ in lowest terms, the denominator $q_n$ satisfies
\begin{equation} \label{eq:qn-growth}
q_n^2>b^{s_n+n+1};
\end{equation}
\item $q_n>bq_{n-1}$;
\item Let $z_n:=b^{s_n}x_n-\lfloor b^{s_n}x_n \rfloor$ be the number in $(0,1)$ whose base $b$ digits are the digits of $x_n$ following the string $0.\hat{a}_1\hat{a}_2,\dots,\hat{a}_n$. Then
\begin{equation} \label{eq:zn-lower-bound}
z_n>\frac{1}{b}+\sum_{i=n}^\infty \frac{1}{b^i}.
\end{equation}
\end{enumerate}

\bigskip
We observe that by our earlier work, these conditions are satisfied for $n=2$, with \eqref{eq:zn-lower-bound} following from the precise lower bound in \eqref{eq:crucial-sandwich}.

Write $x_n=0.c_1 c_2 c_3\dots$ in base $b$. Take
$$l:=\min\{i\geq s_n: (c_{i+1},c_{i+2})\neq (b-1,b-1)\}.$$
This is well defined since $x_n$ does not have a terminating expansion, so the non-terminating expansion of $x_n$ does not end in all $(b-1)$'s. In view of (ii), we can now write
\begin{align*}
x_n&=\frac{p_n}{q_n}=0.c_1\dots c_l \overline{d_1\dots d_m}\\
&=\frac{1}{b^l}\left(\sum_{i=1}^l c_i b^{l-i}+\frac{1}{b^m-1}\sum_{i=1}^m d_i b^{m-i}\right),
\end{align*}
or, after clearing fractions,
\begin{equation} \label{eq:p-and-q-identity}
b^l(b^m-1)p_n=q_n\left[(b^m-1)\sum_{i=1}^l c_i b^{l-i}+\sum_{i=1}^m d_i b^{m-i}\right].
\end{equation}
Now choose $a_{n+1}$ so that $\hat{a}_{n+1}=c_{s_n+1}\dots c_l d_1\dots d_m$, in other words,
\begin{equation} \label{eq:a-n-plus-one}
a_{n+1}=\sum_{i=s_n+1}^l c_i b^{m+l-i}+\sum_{i=1}^m d_i b^{m-i}.
\end{equation}
Thus, $a_{n+1}$ consists of the ``next" $(l-s_n)+m$ digits of $x_n$. Using \eqref{eq:p-and-q-identity} and \eqref{eq:a-n-plus-one}, we can calculate
\begin{align*}
q_{n+1}&=a_{n+1}q_n+q_{n-1}\\
&=q_n\left(b^m\sum_{i=1}^l c_i b^{l-i}+\sum_{i=1}^m d_i b^{l-i}-b^m\sum_{i=1}^{s_n} c_i b^{l-i}\right)+q_{n-1}\\
&=b^l(b^m-1)p_n+q_n\left(\sum_{i=1}^l c_i b^{l-i}-b^m\sum_{i=1}^{s_n} c_i b^{l-i}\right)+q_{n-1}\\
&=b^{l+m}p_n-q_n b^m \sum_{i=1}^{s_n} c_i b^{l-i}+q_n\sum_{i=1}^l c_i b^{l-i}-b^l p_n+q_{n-1}.
\end{align*}
Set
$$A:=b^{l+m}p_n-q_n b^m \sum_{i=1}^{s_n} c_i b^{l-i}, \qquad B:=q_n\sum_{i=1}^l c_i b^{l-i}-b^l p_n+q_{n-1}.$$
Observe that $b^m|A$. We claim that $B\neq 0$. Note first that $d_1\neq 0$: Either $l=s_n$ and $d_1$ is the first digit of $z_n$, so $d_1>0$ by \eqref{eq:zn-lower-bound}; or else $l>s_n$ and $d_1=b-1$ by definition of $l$. Thus, we have
$$\frac{p_n}{q_n}>\sum_{i=1}^l \frac{c_i}{b^i}+\frac{1}{b^{l+1}},$$
which can be rewritten as
$$q_n\sum_{i=1}^l c_i b^{l-i}-b^l p_n<-\frac{q_n}{b}.$$
Hence by (iv), $B<0$. Now we again choose $m$ so large that $\min_{1\leq i\leq r}P_i^{e_i m}>|B|$. Then for each $i$, $P_i^{e_i m}\!\nmid\! q_{n+1}$. Setting $s_{n+1}:=l+m>m$ we conclude, as in the argument following \eqref{eq:no-division}, that the expansion of $x_{n+1}=p_{n+1}/q_{n+1}$ is periodic at least from the $s_{n+1}$-th digit on. (The analog of \eqref{eq:crucial-sandwich} that we need for this is property (v), which we verify for $n+1$ in place of $n$ below.)

Furthermore, choosing $m$ so that $a_{n+1}\geq b^2$ ($m\geq 3$ suffices), we guarantee $q_{n+1}>a_{n+1} q_n>bq_n$ giving (iv) for $n+1$ in place of $n$; and
$$
q_{n+1}^2>a_{n+1}^2 q_n^2 \geq b^2 a_{n+1}q_n^2
\geq b^2 b^{\mathrm{length}(\hat{a}_{n+1})-1}b^{s_n+n+1}=b^{s_{n+1}+(n+1)+1},
$$
yielding (iii) for $n+1$ in place of $n$. Next, (iii) also implies
\begin{equation}
|x_{n+1}-x_n|=\left|\frac{p_{n+1}}{q_{n+1}}-\frac{p_n}{q_n}\right|=\frac{1}{q_n q_{n+1}}
<\frac{1}{a_{n+1}q_n^2}\leq \frac{1}{b^{s_{n+1}+n}}.
\label{eq:consecutive-difference}
\end{equation}
Note that 
$$0.\overline{d_1\dots d_m}>\frac{1}{b}+\sum_{i=n}^\infty \frac{1}{b^i}.$$
This is clear if $d_1=b-1$. Otherwise, $l=s_n$ and $0.\overline{d_1\dots d_m}=z_n$, so the above inequality follows from (v). Since the expansion of $x_n$ begins with $0.c_1\dots c_l d_1\dots d_m d_1\dots d_m$ and $(d_1,d_2)\neq(b-1,b-1)$, it therefore follows from \eqref{eq:consecutive-difference} that the expansion of $x_{n+1}$ begins with $0.c_1\dots c_l d_1\dots d_m$ and moreover,
$$z_{n+1}>z_n-\frac{1}{b^n}>\frac{1}{b}+\sum_{i=n+1}^\infty \frac{1}{b^i}.$$
Thus, we have (i) and (v), and then also (ii), for $n+1$ in place of $n$.

\bigskip
Observe that condition (v) in the induction step guarantees that the first digit in the expansion of $x_n$ following the string $0.\hat{a}_1\dots\hat{a}_n$ is never zero. This is crucial, because if this digit is ever equal to zero, then the process cannot be continued.

We also point out that the sequence $(s_n)$ will in general grow super-exponentially fast. This is so because $q_n$ grows super-exponentially fast since the strings $\hat{a}_n$ are taken ever longer and longer, and the length of the period $d_1\dots d_m$ of $p_n/q_n$ will typically be roughly of the same order of magnitude as $q_n$.

Finally, we note that since at each stage $n$, we have a choice between infinitely many values of $m$, all leading to different choices of $a_n$, it is immediate that there are continuum many Trott numbers.

\begin{remark}
{\rm
It is easy to modify our algorithm above to ensure that it creates only transcendental Trott numbers. The Thue-Siegel-Roth theorem (see \cite{Roth}) says that, if for an irrational number $x$ there exists $\eps>0$ and infinitely many fractions $p/q$ such that
$$\left|x-\frac{p}{q}\right|<\frac{1}{q^{2+\eps}},$$
then $x$ is transcendental. Since for $x=[0;a_1,a_2,\dots]$ 
$$\left|x-\frac{p_n}{q_n}\right|<\frac{1}{q_n q_{n+1}}<\frac{1}{a_{n+1} q_n^2},$$
the Thue-Siegel-Roth condition is satisfied if $a_{n+1}>q_n^\eps$ for infinitely many $n$ and some $\eps>0$. We can always choose $a_{n+1}$ this large, and often the algorithm will force us to, because the string $\hat{a}_{n+1}$ contains at least one full period of the base $b$ expansion of $p_n/q_n$.

An interesting but probably very difficult question is, whether there exist algebraic Trott numbers.
}
\end{remark}

\section{Borel classification of $T_b$} \label{sec:Borel}

In this section we prove Theorem \ref{thm:G-delta}. We will use the following notation. For a finite sequence $(a_1,\dots,a_n)\in\NN^n$, we let $I(a_1,\dots,a_n)$ denote the open interval whose endpoints are $[0;a_1,\dots,a_n]$ and $[0;a_1,\dots,a_{n-1},a_n+1]$. Furthermore, we let $J(a_1,\dots,a_n)$ denote the open interval of numbers in $(0,1)$ whose base $b$ expansion begins with the string $0.\hat{a}_1\dots\hat{a}_n$. Precisely,
$$J(a_1,\dots,a_n):=\left(\sum_{i=1}^n a_i b^{-s_i},\sum_{i=1}^n a_i b^{-s_i}+b^{-s_n}\right),$$
where $s_i$ denotes the length of the string $\hat{a}_1\dots\hat{a}_i$ for $i=1,2,\dots,n$.

\begin{proof}[Proof of Theorem \ref{thm:G-delta}]
Fix $b\in\Gamma$. 
We first show that $T_b$ is a $G_\delta$ set. With any finite sequence $(a_1,\dots,a_n)\in\bigcup_{n=1}^\infty \NN^n$ we associate a set
$$E(a_1,\dots,a_n):=I(a_1,\dots,a_n)\cap J(a_1,\dots,a_n).$$
(For most sequences, this set will be empty.) Since $I(a_1,\dots,a_n,a_{n+1})\subset I(a_1,\dots,a_n)$ and $J(a_1,\dots,a_n,a_{n+1})\subset J(a_1,\dots,a_n)$ it follows that
\begin{equation} \label{eq:nested-Es}
E(a_1,\dots,a_n,a_{n+1})\subset E(a_1,\dots,a_n) \qquad\forall\ (a_1,\dots,a_{n+1})\in\NN^{n+1}.
\end{equation}
Now set
$$G_n:=\bigcup_{(a_1,\dots,a_n)\in\NN^n} E(a_1,\dots,a_n), \qquad n\in\NN$$
and
$$G:=\bigcap_{n=1}^\infty G_n.$$
It is clear that each $G_n$ is open, hence $G$ is a $G_\delta$ set. We claim that $G=T_b$.

First, let $x\in G$. Then for each $n\in\NN$ there is a sequence $(a_1,\dots,a_n)\in\NN^n$ such that $x\in E(a_1,\dots,a_n)$. By \eqref{eq:nested-Es} and the fact that for each $n$, the collection $\{E(a_1,\dots,a_n)\}$ is pairwise disjoint, this means there is an infinite sequence $(a_1,a_2,\dots)$ such that $x\in I(a_1,\dots,a_n)\cap J(a_1,\dots,a_n)$ for each $n$. Hence $x\in T_b$.

Conversely, take $x\in T_b$. Then $x$ is irrational, so it has a unique continuous fraction expansion $(a_1,a_2,\dots)$. This means $x\in I(a_1,\dots,a_n)$ for each $n$. Since $x$ is a Trott number, $x$ must also lie in $J(a_1,\dots,a_n)$ for each $n$. Therefore, $x\in G$.

Next, we show that $T_b$ is not an $F_\sigma$ set. We do this via continuous reduction, by constructing a continuous map $f:2^\omega\to T_b$ such that $f^{-1}(T_b)=H$, where $2^\omega$ denotes the Cantor space of all infinite sequences of $0$'s and $1$'s and $H$ is the set of sequences in $2^\omega$ which contain infinitely many $1$'s. It is well known that $H$ is a complete $G_\delta$ set (see, e.g., \cite[Exercise 23.1]{Kechris}), so the construction of such a map $f$, together with the first half of this proof, will imply that $T_b$ is a complete $G_\delta$ set as well.

Fix an initial segment $(a_1,\dots,a_{n_0})$ as follows. If $b\in\{2,3\}$, we choose $(a_1,\dots,a_{n_0})$ so that
it satisfies the hypotheses of Lemma \ref{lem:once-started} with $n=n_0$ for some $l$. If $b\in\Gamma\backslash\{2,3\}$, we take $n_0=2$ and choose $(a_1,a_2)$ as outlined in the previous section. In both cases, the proofs from the previous section show that $(a_1,\dots,a_{n_0})$ can be extended to an infinite sequence $(a_1,a_2,\dots)$ such that for each $n\geq n_0$ the expansion of $[0;a_1,\dots,a_n]$ is periodic to the right of the string $0.\hat{a}_1\dots\hat{a}_n$ and the number $x=[0;a_1,a_2,\dots]$ is Trott. Furthermore, the string $\hat{a}_{n+1}$ may encompass as many periods of this expansion as we wish.

Now let a sequence $\xi=(\xi_1,\xi_2,\dots)\in 2^\omega$ be given. We begin with $n=n_0$ and $k=0$. Let $(d_1,\dots,d_m)$ be such that $x_n=[0;a_1,\dots,a_n]=0.\hat{a}_1\dots\hat{a}_n\overline{d_1\dots d_m}$. For $i=1,2,\dots$, we do the following. If $\xi_i=0$, we increment $k$ by 1. If $\xi_i=1$, we choose $a_{n+1}$ as outlined in the previous section so that the string $\hat{a}_{n+1}$ encompasses at least $k$ periods of the expansion of $x_n$ (that is, $\hat{a}_{n+1}$ begins with $(d_1\dots d_m)^k$); we increment $n$ by 1 and reset $k$ to $0$. We repeat these steps for each $i\in\NN$ successively. 

If $\xi\not\in H$, then this procedure only generates some finite sequence $(a_1,\dots,a_N)$, and we set $f(\xi)=[0;a_1,\dots,a_N]$. In this case, $f(\xi)$ is rational and so $f(\xi)\not\in T_b$. If on the other hand $\xi\in H$, then the procedure generates an infinite sequence $(a_1,a_2,\dots)$ such that the number $x=[0;a_1,a_2,\dots]$ is Trott, and we set $f(\xi)=x$. 

It follows that $f^{-1}(T_b)=H$, and it is not difficult to see that $f$ is continuous. Thus, $f$ is the desired map, and $T_b$ is a complete $G_\delta$ set.
\end{proof}

\section{The upper bound for Hausdorff dimension} \label{sec:upper-bound}

In this section we prove Theorem \ref{thm:upper-bound}. The only information about Trott numbers that we use here is the following elementary lemma, which follows immediately from the fact that for even $n$, the number $[0;a_1,\dots,a_n,c]$ is decreasing in $c$.

\begin{lemma} \label{lem:one-per-magnitude}
For each even $n$, each sequence $(a_1,\dots,a_n)$ and each $k\in\NN\cup\{0\}$, there is at most one number $c:=c(a_1,\dots,a_n)\in \{b^k,b^k+1,\dots,b^{k+1}-1\}$ such that the interval $I(a_1,\dots,a_n,c)$ intersects $T_b$, where $I(a_1,\dots,a_n,c)$ is defined as in Section \ref{sec:Borel}.
\end{lemma}

\begin{proof}[Proof of Theorem \ref{thm:upper-bound}]
First, we prove that the union $T=\bigcup_{b\geq 2} T_b$ is nowhere dense. It is immediate from Lemma \ref{lem:one-per-magnitude} that $T_b$ is nowhere dense for each $b$. Let $(\alpha,\beta)\subset (0,1)$ be an arbitrary interval; without loss of generality we may assume that $\alpha>0$. Then we can choose $B\in\NN$ so that $1/(\sqrt{B}-1)\leq \alpha$. If $x=[0;a_1,a_2,\dots]$ is Trott in some base $b\geq B$, then $a_1=\lfloor \sqrt{b} \rfloor \geq \sqrt{b}-1\geq \sqrt{B}-1$, and so $x<1/a_1\leq 1/(\sqrt{B}-1)\leq\alpha$. Hence, $(\alpha,\beta)\cap \bigcup_{b\geq B}T_b=\emptyset$. Since furthermore $\bigcup_{2\leq b<B} T_b$ is nowhere dense, there is a subinterval of $(\alpha,\beta)$ that does not intersect 
$T$. This gives the desired result.

Next, we show that $\dim_H T<1$.
Let $f_n(x)=\frac{1}{x+n}$, for $n\in\NN$ and $x\in[0,1]$. For $n,m\in\NN$, write $S_{n,m}:=f_n\circ f_m$. Note that
$$S_{n,m}(x)=\frac{x+m}{n(x+m)+1}=\frac{1}{n}\left(1-\frac{1}{n(x+m)+1}\right),$$
and so
$$S_{n,m}'(x)=\big(n(x+m)+1\big)^{-2}.$$

By Lemma \ref{lem:one-per-magnitude}, we have for each $k\in\NN$ the inclusion
$$T_b\subset \bigcup_{(l_1,m_1),\dots,(l_k,m_k)} S_{n_1,m_1}\circ\dots\circ S_{n_k,m_k}([0,1]),$$
where $l_i\in\NN\cup\{0\}$, $m_i\in\NN$ and
$$n_i:=n_i(l_1,m_1,\dots,l_{i-1},m_{i-1},l_i)\in \{b^{l_i},b^{l_i}+1,\dots,b^{l_i+1}-1\}.$$
It follows that
\begin{align*}
\diam\, S_{n_1,m_1}\circ\dots\circ S_{n_k,m_k}([0,1])&\leq \prod_{i=1}^k \max_{x\in[0,1]}|S_{n_i,m_i}'(x)|\\
&=\prod_{i=1}^k (n_i m_i+1)^{-2}
\leq \prod_{i=1}^k \big(b^{l_i}m_i+1\big)^{-2}.
\end{align*}
Given $\delta>0$, choose $k$ so large that $2^{-2k}<\delta$. Then we get from the above, for any $s>0$,
\begin{align*}
\mathcal{H}_\delta^s(T_b)&\leq \sum_{(l_1,m_1),\dots,(l_k,m_k)} \left(\diam S_{n_1,m_1}\circ\dots\circ S_{n_k,m_k}([0,1])\right)^s\\
&\leq \sum_{(l_1,m_1),\dots,(l_k,m_k)}\prod_{i=1}^k\big(b^{l_i}m_i+1\big)^{-2s}\\
&=\left(\sum_{l=0}^\infty\sum_{m=1}^\infty (b^l m+1)^{-2s}\right)^k.
\end{align*}
Hence $\dim_H T_b\leq s_0$, where $s_0:=s_0(b)$ is the solution of
\begin{equation}
\sum_{l=0}^\infty\sum_{m=1}^\infty (b^l m+1)^{-2s}=1.
\label{eq:double-sum}
\end{equation}
This double series cannot be evaluated in closed form, so we estimate it: For $s>1/2$,
\begin{align*}
\sum_{l=0}^\infty\sum_{m=1}^\infty (b^l m+1)^{-2s}&\leq\sum_{m=1}^\infty (m+1)^{-2s}+\sum_{l=1}^\infty\sum_{m=1}^\infty (b^l m)^{-2s}\\
&=\zeta(2s)-1+\frac{b^{-2s}}{1-b^{-2s}}\zeta(2s)\\
&=\frac{b^{2s}}{b^{2s}-1}\zeta(2s)-1.
\end{align*}
Since the left hand side of \eqref{eq:double-sum} is decreasing in $s$, it follows that $s_0$ is no greater than the root of the equation
\begin{equation}
\frac{b^{2s}}{b^{2s}-1}\zeta(2s)=2.
\label{eq:simpler-equation}
\end{equation}
For $b=10$, for instance, this gives $s_0\leq .8745$, and for $b=3$ it gives $s_0\leq .9493$. In the limit as $b\to\infty$, the root of \eqref{eq:simpler-equation} tends to the root of the simple equation $\zeta(2s)=2$, which is about $.8643$. Only for $b=2$ does \eqref{eq:simpler-equation} not have a root below 1, and the estimate is more delicate. Again, the $l=0$ term of
$$\sum_{l=0}^\infty\sum_{m=1}^\infty (2^l m+1)^{-2s}$$
is equal to $\zeta(2s)-1$. The $l=1$ term is
$$\sum_{m=1}^\infty (2m+1)^{-2s}=(1-2^{-2s})\zeta(2s)-1.$$
The $l=2$ term we estimate by an integral, as follows:
\begin{align*}
\sum_{m=1}^\infty (4m+1)^{-2s}&\leq 5^{-2s}+9^{-2s}+\int_2^\infty (4x+1)^{-2s}dx\\
&=5^{-2s}+9^{-2s}+\frac{9^{1-2s}}{4(2s-1)}.
\end{align*}
Finally, for $l\geq 3$ we estimate again by
$$\sum_{m=1}^\infty (2^l m+1)^{-2s}\leq \sum_{m=1}^\infty (2^l m)^{-2s}=2^{-2ls}\zeta(2s).$$
Thus, all taken together, we arrive at
\begin{align*}
\sum_{l=0}^\infty\sum_{m=1}^\infty (2^l m+1)^{-2s}&\leq (2-2^{-2s})\zeta(2s)-2+\sum_{l=3}^\infty 2^{-2ls}\zeta(2s)\\
&\qquad\qquad+5^{-2s}+9^{-2s}+\frac{9^{1-2s}}{4(2s-1)}\\
&=\left(2-2^{-2s}+\frac{2^{-6s}}{1-2^{-2s}}\right)\zeta(2s)-2+5^{-2s}+9^{-2s}+\frac{9^{1-2s}}{4(2s-1)}.
\end{align*}
Setting this last expression equal to 1 and solving numerically, we obtain $s_0(2)\leq .9979$.

The above argument shows that $\dim_H T_b<1$ for every $b$. Observe also that the upper bounds obtained tend to $.8643$ as $b\to\infty$. It then follows from the countable stability of Hausdorff dimension that $\dim_H T\leq\sup_b \dim_H T_b<1$.
\end{proof}

\begin{remark}
{\rm
Of course, better estimates for $\dim_H T_b$ can be obtained by using more sophisticated methods, or even just by using a version of the above argument that considers four-fold compositions of the maps $f_n$. However, any upper bound we get by using nothing else about Trott numbers than Lemma \ref{lem:one-per-magnitude} is likely to be very far from the actual dimension. We will leave further dimension questions regarding $T_b$ for future work.
}
\end{remark}

\section{Multiple Trott numbers} \label{sec:multiple-Trott}

The goal of this section is to prove Theorem \ref{thm:multiple-Trott}. The crucial element in the proofs below is the following formula for $a_2$.

\begin{lemma} \label{lem:unique-choice}
Let $k\geq 2$ and $k^2+1\leq b\leq k^2+k$. If $b=5$ assume $l\geq 3$; else assume $l\geq 2$.
Suppose $x=[0;a_1,a_2,\dots]\in T_b$ and $a_2$ has $l$ digits in base $b$. Then
\begin{equation} \label{eq:unique-a2}
a_2=\left\lceil \frac{b^l(b-k^2)}{k}-\frac{b}{k(b-k^2)}\right\rceil-1.
\end{equation}
\end{lemma}

\begin{proof}
Let $\alpha$ denote the right hand side of \eqref{eq:unique-a2}. Recall that $a_1=k$ is forced. Hence, in order for $x=[0;a_1,a_2,\dots]$ to lie in $T_b$, the intervals
\begin{equation} \label{eq:overlapping-intervals}
\left[\frac{k}{b}+\frac{a_2}{b^{l+1}},\frac{k}{b}+\frac{a_2+1}{b^{l+1}}\right] \qquad \mbox{and} \qquad
\left[\frac{a_2}{ka_2+1},\frac{a_2+1}{k(a_2+1)+1}\right]
\end{equation}
must overlap. We first show that, if $a_2\geq\alpha+1$, then
\begin{equation} \label{eq:to-the-right}
\frac{k}{b}+\frac{a_2}{b^{l+1}}\geq \frac{a_2+1}{k(a_2+1)+1},
\end{equation}
prohibiting the overlap. Set $z:=ka_2$. Clearing fractions in \eqref{eq:to-the-right}, multiplying by $k$ and rearranging, we see that \eqref{eq:to-the-right} is equivalent to the quadratic inequality
\begin{equation} \label{eq:quadratic-1}
z^2-\{b^l(b-k^2)-(k+1)\}z+b^l k(k^2+k-b)\geq 0.
\end{equation}
Now we observe from \eqref{eq:unique-a2} that
\begin{equation} \label{eq:alpha-sandwich}
k\alpha<b^l(b-k^2)-\frac{b}{b-k^2}\leq k(\alpha+1),
\end{equation}
so $k\alpha$ lies to the right of the vertex of the parabola on the left of \eqref{eq:quadratic-1}, as can be checked readily by using that $b\geq k^2+1$ and $l\geq 2$. Therefore, we obtain, using the second inequality in \eqref{eq:alpha-sandwich},
\begin{align*}
z^2-&\{b^l(b-k^2)-(k+1)\}z+b^l k(k^2+k-b) \\
&\geq k^2(\alpha+1)^2-\{b^l(b-k^2)-(k+1)\}k(\alpha+1)+b^l k(k^2+k-b) \\
&\geq \left\{b^l(b-k^2)-\frac{b}{b-k^2}\right\}^2-\{b^l(b-k^2)-(k+1)\}\left\{b^l(b-k^2)-\frac{b}{b-k^2}\right\}\\
& \qquad\qquad\qquad\qquad\qquad +b^l k(k^2+k-b)\\
&=-b^{l+1}+\left(\frac{b}{b-k^2}\right)^2+b^l(b-k^2)(k+1)-\frac{b(k+1)}{b-k^2}+b^l k(k^2+k-b)\\
&=\left(\frac{b}{b-k^2}\right)^2-\frac{b(k+1)}{b-k^2}=\frac{b}{b-k^2}\left(\frac{b}{b-k^2}-(k+1)\right)\\
&=\frac{b}{b-k^2}\left(\frac{k^2}{b-k^2}-k\right)=\frac{bk}{b-k^2}\left(\frac{k}{b-k^2}-1\right)\\
&\geq 0,
\end{align*}
where the last inequality follows since $b\leq k^2+k$. This proves \eqref{eq:quadratic-1}, and hence \eqref{eq:to-the-right}.

Next, we show that if $a_2\leq \alpha-1$, then
\begin{equation} \label{eq:to-the-left}
\frac{k}{b}+\frac{a_2+1}{b^{l+1}}\leq \frac{a_2}{ka_2+1}. 
\end{equation}
Again setting $z:=k a_2$ and clearing fractions, this is equivalent to
\begin{equation} \label{eq:quadratic-2}
z^2-\{b^l(b-k^2)-(k+1)\}z+k(kb^l+1)\leq 0.
\end{equation}
Since $b^{l-1}\leq a_2\leq\alpha-1$ and the expression on the left side of \eqref{eq:quadratic-2} is convex in $z$, it is enough to verify \eqref{eq:quadratic-2} for $z=kb^{l-1}$ and $z=k(\alpha-1)$. Substituting $z=kb^{l-1}$ we obtain
\begin{align*}
k^2 b^{2l-2}-&\{b^l(b-k^2)-(k+1)\}kb^{l-1}+k^2 b^l+k\\
&=k\big[b^{2l-2}\{k-b(b-k^2)\}+(k+1)b^{l-1}+kb^l+1\big]\\
&\leq k\big[-(k^2-k+1)b^{2l-2}+(k+1)b^l\big]\\
&=-kb^l\big[(k^2-k+1)b^{l-2}-(k+1)\big]\\
&\leq -kb^l(k^2-2k)\leq 0.
\end{align*}
Here the first inequality follows since $b\geq k^2+1$ and $(k+1)b^{l-1}<b^l$; the second inequality since $l\geq 2$; and the last inequality since $k\geq 2$.

Next, set $z=k(\alpha-1)$. We consider separately the cases $b=k^2+1$ and $b\geq k^2+2$. In the second case, we slightly strengthen the first inequality in \eqref{eq:alpha-sandwich}. Multiplying by $b-k^2$ and using that $\alpha$ is an integer, we obtain
\[
k\alpha\leq b^l(b-k^2)-\frac{b+1}{b-k^2}.
\]
As a result, 
\begin{align*}
z^2-&\{b^l(b-k^2)-(k+1)\}z+k(kb^l+1)\\
&\leq \left\{b^l(b-k^2)-\frac{b+1}{b-k^2}-k\right\}^2\\
& \qquad\qquad -\{b^l(b-k^2)-(k+1)\}\left\{b^l(b-k^2)-\frac{b+1}{b-k^2}-k\right\}+k(kb^l+1)\\
&=-b^l(b+1)-b^l k(b-k^2)+\left(\frac{b+1}{b-k^2}+k\right)^2+b^l(b-k^2)(k+1)\\
& \qquad\qquad -(k+1)\left(\frac{b+1}{b-k^2}+k\right)+k(kb^l+1)\\
&=-b^l+\left(\frac{b+1}{b-k^2}+k\right)^2-(k+1)\left(\frac{b+1}{b-k^2}+k\right)+k\\
&=-b^l+\left(\frac{k^2+1}{b-k^2}+k+1\right)\cdot \frac{k^2+1}{b-k^2}+k\\
&\leq -(k^2+2)^2+\left(\frac{k^2+1}{2}+k+1\right)\cdot\frac{k^2+1}{2}+k\\
&\leq 0,
\end{align*}
where we used that $l\geq 2$ and $b\geq k^2+2$, and the last inequality is easily verified using basic algebra.

When $b=k^2+1$, we need to be slightly more precise still. In this case, we have exactly $k\alpha=b^l-(k^2+k+1)$ from \eqref{eq:unique-a2}, so $z=k\alpha-k=b^l-(k+1)^2$ and
\begin{align*}
z^2-&\{b^l(b-k^2)-(k+1)\}z+k(kb^l+1)\\
&=\{b^l-(k+1)^2\}^2-\{b^l-(k+1)\}\{b^l-(k+1)^2\}+k(kb^l+1)\\
&=-b^l(k+1)^2+(k+1)^4+b^l(k+1)-(k+1)^3+k^2 b^l+k\\
&=-kb^l+(k+1)^4-(k+1)^3+k\\
&\leq -k(k^2+1)^2+(k+1)^4-(k+1)^3+k\\
&=k\{-(k^2+1)^2+(k+1)^3+1\}=-k(k^4-k^3-k^2-3k-1)\\
&\leq 0,
\end{align*}
where the last inequality holds for all $k\geq 3$. It remains to deal with the case $k=2$, i.e. $b=5$. Here $l\geq 3$ by the assumption of the lemma, and the calculation simplifies: $z=5^l-9$, and so
\begin{align*}
z^2-\{b^l(b-k^2)&-(k+1)\}z+k(kb^l+1)\\
&=(5^l-9)^2-(5^l-3)(5^l-9)+4\cdot 5^l+2\\
&=-2\cdot 5^l+56\leq 0
\end{align*}
for all $l\geq 3$.
\end{proof}

\begin{remark}
{\rm
\begin{enumerate}[(a)]
\item The conclusion of the lemma fails for $l=1$. For instance, when $b=6$, it appears that $a_2$ can be either 1, 2, 3 or 4, although we do not know whether for each of these choices of $a_2$ the sequence of partial quotients can be continued indefinitely. In addition, if $b=5$ and $l=2$, then \eqref{eq:unique-a2} prescribes $a_2=9$, but $a_2=8$ appears to be possible as well, since the intervals in \eqref{eq:overlapping-intervals} overlap for both values of $a_2$.
\item The expression for $a_2$ dictated by \eqref{eq:unique-a2} is consistent with our choice of $a_2$ in \eqref{eq:a2}, as may be seen by replacing $l$ with $l+m-1$ in \eqref{eq:unique-a2} and doing some simple algebra.
\end{enumerate}
}
\end{remark}

In a few special cases, the expression in \eqref{eq:unique-a2} simplifies.

\begin{corollary} \label{cor:a2-special-cases}
Let $a_2$ have $l$ digits in base $b$, where $l\geq 2$.
\begin{enumerate}[(i)]
\item If $b=k^2+1$ with $k\geq 3$, then
\begin{equation} \label{eq:base10-etc}
a_2=\frac{(k^2+1)^l-(k^2+1)}{k}-1.
\end{equation}
This expression is also valid when $b=5$ and $l\geq 3$.
\item If $b=k^2+k$ with $k\geq 2$, then
\begin{equation} \label{eq:the-upper-extreme}
a_2=b^l-2=(k^2+k)^l-2.
\end{equation}
\end{enumerate}
\end{corollary}

Bases 2 and 3 require special consideration. 

\begin{lemma} \label{lem:a2-for-bases-2-and-3}
\begin{enumerate}[(i)]
\item If $b=2$, then either $a_2=4$ or $a_2=2^l-3$ for some $l\geq 3$.
\item If $b=3$, then necessarily $a_2=1$.
\end{enumerate}
\end{lemma}

\begin{proof}
(i) Let $b=2$, and suppose $a_2$ has $l$ binary digits. In order for the intervals in \eqref{eq:overlapping-intervals} to overlap, it is necessary on the one hand that $a_2^2-(2^l-2)a_2<0$, so that $a_2<2^l-2$; this implies $l\geq 3$. On the other hand, we must have $a_2^2-(2^l-2)a_2+2^l\geq 0$. When $l=3$, these inequalities are satisfied for $a_2=4, 5$. When $l\geq 4$, the last inequality is satisfied if and only if $a_2\geq 2^l-3$, bearing in mind that $a_2\geq 2^{l-1}$, and hence $a_2$ lies to the right of the vertex of the parabola $y=x^2-(2^l-2)x+2^l$. This gives the result for $b=2$.

(ii) Next, let $b=3$ and suppose $a_2$ has $l$ ternary digits. Note that here $a_1=1$. If the intervals in \eqref{eq:overlapping-intervals} overlap, we have in particular
\[
\frac{a_2}{a_2+1}<\frac{1}{3}+\frac{a_2+1}{3^{l+1}},
\]
which implies $a_2^2-2(3^l-1)a_2+3^l\geq 0$. Since $3^{l-1}\leq a_2<3^l$, $a_2$ lies to the left of the vertex of the parabola in this last inequality. Hence
\[
a_2^2-(2\cdot 3^l-2)a_2+3^l\leq 3^{2l-2}-2(3^l-1)3^{l-1}+3^l=-5\cdot 3^{l-1}(3^{l-1}-1),
\]
where the inequality is strict if $l=1$ and $a_2=2$. This gives a contradiction, unless $l=a_2=1$.
\end{proof}

\begin{lemma} \label{lem:more-digits}
If $x=[0;a_1 a_2\dots]$ is Trott in both base $b$ and in base $c>b$, then $b$ and $c$ must belong to the same interval $[k^2+1,k^2+k]$ for some $k$. Furthermore, if $a_2$ has $l$ digits in base $b$ and $m$ digits in base $c$, then $2\leq m<l$.
\end{lemma}

\begin{proof}
The first statement is clear, since $a_1=k$ is forced. It is also clear that any integer has at least as many digits in base $b$ as in base $c$. Suppose $a_2$ has $l$ digits in both bases. Then $x$ lies in both of the intervals
\[
\left[\frac{k}{b}+\frac{a_2}{b^{l+1}},\frac{k}{b}+\frac{a_2+1}{b^{l+1}}\right] \qquad\mbox{and} \qquad
\left[\frac{k}{c}+\frac{a_2}{c^{l+1}},\frac{k}{c}+\frac{a_2+1}{c^{l+1}}\right].
\]
Hence these intervals must overlap. However, we claim that
\[
\frac{k}{c}+\frac{a_2+1}{c^{l+1}}<\frac{k}{b}+\frac{a_2}{b^{l+1}}.
\]
This follows since
\begin{align*}
\frac{k}{b}+\frac{a_2}{b^{l+1}}-\left(\frac{k}{c}+\frac{a_2+1}{c^{l+1}}\right)
&=\frac{k(c-b)}{bc}+a_2\left(\frac{1}{b^{l+1}}-\frac{1}{c^{l+1}}\right)-\frac{1}{c^{l+1}}\\
&\geq \frac{1}{bc}-\frac{1}{c^2}>0.
\end{align*}
This contradiction shows that $l>m$. Finally, we argue that $m$ cannot be 1. Since $l\geq 2$ and $b\geq k^2+1$, we have
\begin{align*}
a_2&\geq\left\lceil \frac{b^2(b-k^2)}{k}-\frac{b}{k(b-k^2)} \right\rceil-1
=\left\lceil \frac{b^l(b-k^2)}{k}-\frac{1}{k}-\frac{k}{b-k^2} \right\rceil-1\\
&\geq \frac{(k^2+1)^2-(k^2+1)}{k}-1=k^3+k-1\geq k^2+k\geq c,
\end{align*}
and it follows that $m\geq 2$.
\end{proof}

Note that it is necessary to look at $a_2$, since the first stage intervals
$\big[\frac{k}{b},\frac{k+1}{b}\big]$ and $\big[\frac{k}{c},\frac{k+1}{c}\big]$
always overlap. This follows because
\[
\frac{c}{b}\leq\frac{k^2+k}{k^2+1}=1+\frac{k-1}{k^2+1}<1+\frac{1}{k},
\]
and hence $\frac{k}{b}<\frac{k+1}{c}$. The other needed inequality is trivial.

\bigskip
We are now ready to start proving Theorem \ref{thm:multiple-Trott}. We begin with the following special case:

\begin{proposition} \label{prop:both-extremes}
Let $k\geq 2$. No number is Trott in both base $k^2+1$ and base $k^2+k$.
\end{proposition}

\begin{proof}
If $a_2$ has $l$ digits in base $k^2+1$ and $m$ digits in base $k^2+k$, then $l\geq 3$ by Lemma \ref{lem:more-digits}, so Corollary \ref{cor:a2-special-cases} yields
\[
\frac{(k^2+1)^l-(k^2+1)}{k}-1=(k^2+k)^m-2.
\]
Hence, $(k^2+1)^l-k^2-1=k(k^2+k)^m-k$. But this is impossible, since both the left hand side and $k(k^2+k)^m$ are divisible by $k^2$, whereas $-k$ clearly is not.
\end{proof}

The next lemma develops two basic inequalities that we use repeatedly in the sequel.

\begin{lemma} \label{lem:sandwich-inequalities}
Let $k^2+1\leq b<c\leq k^2+k$ for $k\geq 2$, and suppose $[0;k,a_2,a_3,\dots]\in T_b\cap T_c$. Let $a_2$ have $l$ digits in base $b$, and $m$ digits in base $c$. Then
\begin{equation} \label{eq:difference-sandwich}
\frac{k^2(c-b)}{(b-k^2)(c-k^2)}-k<b^l(b-k^2)-c^m(c-k^2)<\frac{k^2(c-b)}{(b-k^2)(c-k^2)}+k,
\end{equation}
and in particular,
\begin{equation} \label{eq:k-square-bound}
\big|b^l(b-k^2)-c^m(c-k^2)\big|<k^2.
\end{equation}
\end{lemma}

\begin{proof}
By \eqref{eq:unique-a2} we have
\[
\left|\frac{b^l(b-k^2)}{k}-\frac{b}{k(b-k^2)}-\left(\frac{c^m(c-k^2)}{k}-\frac{c}{k(c-k^2)}\right)\right|<1,
\]
and so
\[
\left|b^l(b-k^2)-c^m(c-k^2)-\left(\frac{b}{b-k^2}-\frac{c}{c-k^2}\right)\right|<k.
\]
Note
\[
\frac{b}{b-k^2}-\frac{c}{c-k^2}=k^2\left(\frac{1}{b-k^2}-\frac{1}{c-k^2}\right)=\frac{k^2(c-b)}{(b-k^2)(c-k^2)}.
\]
This gives \eqref{eq:difference-sandwich}. Now notice that the middle expression above is increasing in $c$ and decreasing in $b$, so it is maximized when $b=k^2+1$ and $c=k^2+k$. Hence
\[
\frac{k^2(c-b)}{(b-k^2)(c-k^2)}+k\leq \frac{k^2(k-1)}{1\cdot k}+k=k^2.
\]
This proves \eqref{eq:k-square-bound}.
\end{proof}

\begin{proposition} \label{prop:one-apart}
No number is Trott in bases $b$ and $b+1$, for any $b$.
\end{proposition}

\begin{proof}
Of course we may assume that $b,b+1\in\Gamma$, as otherwise the statement is trivial.
We first treat two special cases. When $b=2$, the statement follows immediately from Lemma \ref{lem:a2-for-bases-2-and-3}, and when $b=5$ it follows from Proposition \ref{prop:both-extremes}.

We may henceforth assume that $k^2+1\leq b<k^2+k$ with $k\geq 3$. Let $a_2$ have $l$ digits in base $b$ and $m$ digits in base $b+1$. Setting $c=b+1$ in \eqref{eq:difference-sandwich} we obtain
\[
\left|b^l(b-k^2)-(b+1)^m(b+1-k^2)\right|<\frac{k^2}{(b-k^2)(b+1-k^2)}+k\leq \frac{k^2}{2}+k.
\]
Let $j:=b-k^2$; then the above can be written as
\[
|jb^l-(j+1)(b+1)^m|<\frac{k^2}{2}+k.
\]
Now $jb^l-(j+1)(b+1)^m$ is congruent to $-(j+1)$ modulo $b$, and to $(-1)^l j$ modulo $b+1$. Here we have $1\leq j<k$. Observe that, since $k\geq 3$, 
\[
-(j+1)-b<-k^2\leq -\left(\frac{k^2}{2}+k\right), 
\]
and
\[
-(j+1)+b=k^2-1>\frac{k^2}{2}+k.
\]
Therefore, it must be the case that $jb^l-(j+1)(b+1)^m=-(j+1)$. But $-(j+1)\not\equiv \pm j \mod (b+1)$, since $2j+1<2k<k^2<b+1$. This contradiction completes the proof.
\end{proof}

In an analogous vein, for bases that are two apart we have the following partial result.

\begin{proposition} \label{prop:difference-of-two}
Let $k\geq 3$ and $1\leq j\leq k-2$, and suppose 
\begin{equation} \label{eq:j-limit}
j\leq\sqrt{\frac{2k^2}{k-2}+1}-1.
\end{equation}
Then no number is Trott in both base $k^2+j$ and $k^2+j+2$.
\end{proposition}

\begin{proof}
Suppose the conclusion of the proposition is false. Then by \eqref{eq:difference-sandwich} we would have
\begin{equation} \label{eq:sandwich-difference-2}
\frac{2k^2}{j(j+2)}-k<j(k^2+j)^l-(j+2)(k^2+j+2)^m<\frac{2k^2}{j(j+2)}+k,
\end{equation}
where $l$ is the number of digits of $a_2$ in base $k^2+j$, and $m$ the number of digits of $a_2$ in base $k^2+j+2$. 

Now observe that $j(k^2+j)^l-(j+2)(k^2+j+2)^m$ is congruent to $j(-1)^l-(j+2)$ modulo $k^2+j+1$, hence is either $-2$ or $-2j-2$ modulo $k^2+j+1$. But \eqref{eq:j-limit}, together with the left half of \eqref{eq:sandwich-difference-2} implies that $j(k^2+j)^l-(j+2)(k^2+j+2)^m>-2$, so it must be the case that 
\[
j(k^2+j)^l-(j+2)(k^2+j+2)^m\geq -2j-2+(k^2+j+1)=k^2-j-1.
\]
We will now show that this, too is impossible. If $j\geq 2$, then 
\begin{equation*} 
k^2-j-1\geq k^2-k\geq \frac{k^2}{4}+k \geq \frac{2k^2}{j(j+2)}+k,
\end{equation*}
contradicting the right half of \eqref{eq:sandwich-difference-2}. This leaves the case $j=1$. Here 
\[
k^2-j-1=k^2-2\geq \frac{2k^2}{3}+k=\frac{2k^2}{j(j+2)}+k,
\]
provided that $k\geq 5$, in which case we again have a contradiction. The case $k=3$, $j=1$ (bases 10 and 12) is covered by Proposition \ref{prop:both-extremes}, so it only remains to check the case $k=4$, $j=1$ (i.e. bases 17 and 19). Here we can use \eqref{eq:unique-a2} directly to find that
\[
a_2=\frac{17^l-21}{4}=\begin{cases}
\frac{3\cdot 19^m-9}{4} & \mbox{if $m$ is odd},\\
\frac{3\cdot 19^m-7}{4} & \mbox{if $m$ is even},
\end{cases}
\]
and hence
\[
17^l-3\cdot 19^m=\begin{cases}
12 & \mbox{if $m$ is odd},\\
14 & \mbox{if $m$ is even}.
\end{cases}
\]
We can immediately rule out 12 by considering both sides modulo 3. But 14 does not work either: considering first the equation modulo 6 we see that $l$ must be odd. Now consider the equation modulo 19, which gives $(-2)^l\equiv -5 \mod 19$, and since $l$ is odd, this becomes $2^l\equiv 5 \mod 19$. However, computing $2,2^3,2^5,\dots,2^{17}$ modulo 19 we get $2,8,13,14,18,15,3,12,10$, and higher powers simply repeat this pattern. So there is no odd $l$ such that $2^l\equiv 5 \mod 19$.
\end{proof}

Our last two results require the following lower bound on $m$.

\begin{lemma} \label{lem:m-lower-bound}
Assume the hypotheses of Lemma \ref{lem:sandwich-inequalities}. Then $m>k\log k$.
\end{lemma}

\begin{proof}
From \eqref{eq:k-square-bound} we have $c^m(c-k^2)>b^l(b-k^2)-k^2$.
Lemma \ref{lem:more-digits} implies $l\geq 3$, and so $b^l\geq b^3>k^6=k^4 k^2$. It follows that
\[
c^m(c-k^2)>b^l(b-k^2-k^{-4})\geq b^{m+1}(b-k^2-k^{-4}),
\]
and so
\[
\left(\frac{c}{b}\right)^m>\frac{b(b-k^2-k^{-4})}{c-k^2}.
\]
Hence, taking logs and writing $\frac{c}{b}=1+\frac{c-b}{b}$, we obtain
\begin{align*}
m&>\frac{\log b+\log(b-k^2-k^{-4})-\log(c-k^2)}{\log\left(1+\frac{c-b}{b}\right)}\\
&>\frac{b}{c-b}\left(\log b+\log(b-k^2-k^{-4})-\log(c-k^2)\right)\\
&\geq \frac{k^2+1}{k}\left(\log(k^2+1)+\log(1-k^{-4})-\log k\right)\\
&=\frac{k^2+1}{k}\log\left((1-k^{-4})\frac{k^2+1}{k}\right)\\
&>k\log k,
\end{align*}
where the second inequality follows since $\log(1+t)<t$ for all $t>0$; the third inequality follows since $b-k^2\geq 1$ and $c-k^2\leq k$; and the last inequality follows since $(1-k^{-4})(k^2+1)>k^2$ for all $k\geq 2$.
\end{proof}

\begin{proposition} \label{thm:common-factor}
If $\gcd(b,c)>1$, then $T_b\cap T_c=\emptyset$.
\end{proposition}

\begin{proof}
Suppose $x=[0;a_1 a_2\dots]$ is Trott in both base $b$ and base $c$. Fix $k\in\NN$ such that $k^2+1\leq b<c\leq k^2+k$. Since $\gcd(b,c)>1$, it follows that $k\geq 3$.
Let $l$ and $m$ be the number of digits of $a_2$ in base $b$, resp.~base $c$. Then $l>m$.

First, we show that $b^l(b-k^2)-c^m(c-k^2)\neq 0$. Suppose otherwise. Let $d=\gcd(b,c)$. Then
\[
b^{l-m}\left(\frac{b}{d}\right)^m(b-k^2)=\left(\frac{c}{d}\right)^m(c-k^2).
\]
Since $c<2b$, $b$ does not divide $c$ and so $d\neq b$. Hence $b/d$ has a prime factor $p$. But $p\nmid(c/d)$, so $p^m|c-k^2$ and therefore $p^m\leq c-k^2\leq k$. Thus,
\[
m\leq \frac{\log k}{\log p}\leq \frac{\log k}{\log 2}<k\log k.
\]
This contradicts Lemma \ref{lem:m-lower-bound}, and hence $b^l(b-k^2)-c^m(c-k^2)\neq 0$.

Now let $p$ be a common prime factor of $b$ and $c$. Since $m<l$, $p^m$ divides $b^l(b-k^2)-c^m(c-k^2)$. On the other hand, from the above argument and \eqref{eq:k-square-bound} we have $0<|b^l(b-k^2)-c^m(c-k^2)|<k^2$. It follows that $p^m<k^2$, and so
\[
m<\frac{\log(k^2)}{\log p}\leq \frac{2\log k}{\log 2}\leq k\log k,
\]
where the last inequality holds since $k\geq 3$. This again contradicts Lemma \ref{lem:m-lower-bound}.
\end{proof}


When $b$ and $c$ are large enough, we get the desired result without any further assumptions:

\begin{proposition} \label{prop:very-large-b}
Let $c>b>1.185\times 10^{29}$. Then $T_b\cap T_c=\emptyset$.
\end{proposition}

The proof of this proposition uses a theorem of Matveev \cite{Matveev}, which is a strengthening of Baker's theorem (see, e.g. \cite{Baker-Wustholz}). These theorems give lower bounds for the absolute value of expressions of the form
\begin{equation} \label{eq:log-linear-form}
\Lambda:=\beta_1\log\alpha_1+\beta_2\log\alpha_2+\dots+\beta_n\log\alpha_n
\end{equation}
provided $\Lambda\neq 0$, where $\alpha_1,\dots,\alpha_n$ are algebraic numbers and $\beta_1,\dots,\beta_n$ are rational integers. Since we only need Matveev's result for the case when $\alpha_1,\dots,\alpha_n$ are in fact rational numbers, and the general theorem is a bit technical, we state here only the special case that we will use to prove Proposition \ref{prop:very-large-b}. For $\alpha\in\QQ$, let $h(\alpha)$ denote the logarithmic Weil height of $\alpha$; that is, $h(\alpha)=\log\max\{|p|,|q|\}$ if $\alpha=p/q$ in lowest terms.

\begin{lemma}[Matveev] \label{lem:Matveev}
Let $\alpha_1,\dots,\alpha_n$ be rational numbers, not 0 or 1, and $\beta_1,\dots,\beta_n$ be integers. Let $\Lambda$ be defined by \eqref{eq:log-linear-form} and assume $\Lambda\neq 0$. Then
\[
|\Lambda|\geq\exp\{-C_n h(\alpha_1)\cdots h(\alpha_n)\log(eB)\},
\]
where $B:=\max\{|\beta_1|,\dots,|\beta_n|\}$ and
\[
C_n:=\min\left\{\frac{e}{2}\cdot 30^{n+3}n^{4.5},2^{6n+20}\right\}.
\]
\end{lemma}

\begin{proof}[Proof of Proposition \ref{prop:very-large-b}]
We use Matveev's theorem in much the same way as in He and Togb\'e \cite{He-Togbe}. Assume $T_b\cap T_c\neq\emptyset$; then there is an integer $k\geq 3$ such that $k^2<b<c\leq k^2+k$. Suppose $x=[0;a_1,a_2,\dots]\in T_b\cap T_c$. Let $l$ and $m$ be the number of digits of $a_2$ in base $b$ and $c$, respectively. 
Recall from the proof of Proposition \ref{thm:common-factor} that $b^l(b-k^2)-c^m(c-k^2)\neq 0$. Assume without loss of generality that $c^m(c-k^2)>b^l(b-k^2)$. Let
\[
\Lambda:=m\log c-l\log b+\log\left(\frac{c-k^2}{b-k^2}\right).
\]
Applying Lemma \ref{lem:Matveev} with $\alpha_1=c$, $\alpha_2=b$, $\alpha_3=(c-k^2)/(b-k^2)$, $\beta_1=m$, $\beta_2=-l$ and $\beta_3=1$ we obtain
\begin{align*}
\frac{c^m(c-k^2)}{b^l(b-k^2)}-1&=e^{\Lambda}-1>\Lambda\\
&\geq \exp\{-C_3 \log b \log c \log(c-k^2) \log(el)\}\\
&\geq \exp\{-C_3 \big(\log(k^2+k)\big)^2 \log k \log(el)\}=:A_k(l).
\end{align*}
Combining this with \eqref{eq:k-square-bound}, we deduce that
\[
k^2>c^m(c-k^2)-b^l(b-k^2)\geq A_k(l) b^l(b-k^2)\geq A_k(l)(k^2+1)^l>A_k(l)k^{2l},
\]
and hence, $\log A_k(l)+2(l-1)\log k<0$. Observe that $C_3=(e/2)30^6 3^{4.5}\approx 1.3901\times 10^{11}$. Let $C_3':=1.3902\times 10^{11}$. Since our assumption on $b$ implies $k>3.4422\times 10^{14}$ and $l\geq k\log k+1$ by Lemma \ref{lem:m-lower-bound}, it is easy to see that
\[
A_k(l)\geq \exp\{-4C_3'(\log k)^3 \log l\}.
\]
Thus, we obtain
\[
2(l-1)\log k<4C_3'(\log k)^3\log l,
\]
which can be written as
\begin{equation} \label{eq:l-bound}
\frac{l-1}{\log l}<2C_3'(\log k)^2.
\end{equation}
Here the left hand side is increasing in $l$, and since $l\geq k\log k+1$, it follows that
\[
2C_3'(\log k)^2>\frac{k\log k}{\log(k\log k+1)},
\]
and therefore,
\[
2C_3'\log k\log(k\log k+1)>k.
\]
However, one verifies numerically that this inequality fails for $k\geq 3.4422\times 10^{14}$.
\end{proof}

\begin{remark}
{\rm
For $k<3.4422\times 10^{14}$, \eqref{eq:l-bound} gives an implicit upper bound for $l$. Thus, theoretically, the verification of the conjecture is now reduced to a finite number of operations, as there are only finitely many pairs $(b,c)$ left to check, and for each such pair, only finitely many pairs $(l,m)$ to try. However, the number is too large for such a brute force strategy to be feasible.
}
\end{remark}

\begin{proof}[Proof of Theorem \ref{thm:multiple-Trott}]
The theorem follows from Propositions \ref{prop:both-extremes}, \ref{prop:one-apart}, \ref{thm:common-factor}, \ref{prop:difference-of-two} and \ref{prop:very-large-b}.  
\end{proof}

\section*{Acknowledgements}
This work grew from an undergraduate research project at the University of North Texas. We thank Professor Lior Fishman for many helpful discussions and suggestions. We also thank Mercedes Byberg for finding examples of Trott numbers, and Pranoy Dutta for writing code to search for Trott Numbers. These examples were instrumental in beginning this research.


\begin{thebibliography}{5}

\bibitem{Baker-Wustholz}
{\sc A. Baker} and {\sc G. W\"ustholz}, Logarithmic forms and group varieties. {\em J. Reine Angew. Math.} {\bf 442} (1993), 19--62.

\bibitem{He-Togbe}
{\sc B. He} and {\sc A. Togb\'e}, On the number of solutions of the Diophantine equation $ax^m-by^n=c$. {\em Bull. Aust. Math. Soc.} {\bf 81} (2010), 177--185.

\bibitem{Kechris}
{\sc A. S. Kechris}, {\em Classical descriptive set theory}. Graduate Texts in Mathematics, 156. Springer-Verlag, New York, 1995.

\bibitem{Matveev}
{\sc E. M. Matveev}, An explicit lower bound for a homogeneous rational linear form in logarithms of
algebraic numbers II, {\em Izv. Math.} {\bf 64} (2000), 1217--1269.

\bibitem{OEIS}
The Online Encyclopedia of Integer Sequences, sequence A039662; https://oeis.org/A039662.

\bibitem{Roth}
{\sc K. F. Roth}, Rational approximations to algebraic numbers. {\em Mathematika} {\bf 2} (1955), 1--20; corrigendum, 168. 

\bibitem{Trott}
{\sc M. Trott}, Finding Trott Constants. {\em Mathematica J.} {\bf 10} (2006), 303--322. 

\end{thebibliography}
\end{document}